\documentclass[a4paper]{article}
\usepackage{a4}
\usepackage{amsmath,amssymb,amsthm}
\usepackage{dsfont}

\newcommand{\R}{\mathbb R}
\newcommand{\N}{\mathbb N}
\newcommand{\C}{\mathbb C}
\newcommand{\A}{\mathcal{A}}
\newcommand{\B}{\mathcal{B}}
\newcommand{\HH}{\mathcal{H}}
\newcommand{\T}{\mathcal{T}}
\newcommand{\al}{\alpha}
\newcommand{\de}{\delta}
\newcommand{\D}{\Delta}
\newcommand{\la}{\lambda}
\newcommand{\io}{\iota}
\newcommand{\s}{\sigma}
\newcommand{\Om}{\Omega}
\newcommand{\om}{\omega}
\newcommand{\e}{\eta}
\newcommand{\ep}{\varepsilon}
\newcommand{\del}{\partial}
\newcommand{\ti}[1]{\tilde{#1}}
\newcommand{\tid}[1]{\overline{#1}}
\newcommand{\1}{\mathds{1}}
\newcommand{\sle}{{\rm SLE}}
\newcommand{\2}{\tfrac{1}{2}}

\newcommand{\hc}[1]{\boxminus #1}
\newcommand{\p}{\mathbf{p}}

\theoremstyle{plain}
 \newtheorem{thm}{Theorem}
 \newtheorem{lem}[thm]{Lemma}
 \newtheorem{corol}[thm]{Corollary}
\theoremstyle{definition}
 
 \newtheorem*{cond}{Conditions}
 
\DeclareMathOperator{\Var}{Var}

\title{Singularity of Full Scaling Limits of Planar Nearcritical Percolation}
\author{Simon Aumann\\[0.5ex] \textit{\small{Mathematisches Institut, Ludwig-Maximilians-Universit\"at M\"unchen}}\\[-0.5ex] \textit{\small{Theresienstr.\ 39, D-80333 M\"unchen, Germany}}\\[-0.5ex] \small{aumann@math.lmu.de}}

\begin{document}

\maketitle

\begin{abstract}
 We consider full scaling limits of planar nearcritical percolation in the Quad-Crossing-Topology introduced by Schramm and Smirnov. We show that two nearcritical scaling limits with different parameters are singular with respect to each other. The results hold for percolation models on rather general lattices, including bond percolation on the square lattice and site percolation on the triangular lattice.
\end{abstract}

\begin{small}
\noindent
\textit{AMS Mathematics Subject Classification 2010:} 60K35, 82B43, 60G30, 82B27\\
\textit{Keywords:} nearcritical, percolation, full scaling limit, singular
\end{small}

\section{Introduction}

Percolation theory has attracted more and more attention since Smirnov's proof of the conformal invariance of critical percolation interfaces on the triangular lattice. This was the missing link for the existence of a unique scaling limit of critical exploration paths. In the sequel, not only limits of exploration paths, but also limits of full percolation configurations have been explored. To obtain a scaling limit, one considers percolation on a lattice with mesh size $\e>0$ and lets $\e$ tend to $0$. In the case of the full configuration limit, it is a-priori not clear, in what sense, or in what topology, the limit $\e\to0$ shall be taken. There are several possibilities, nine of them are explained in \cite[p.\ 1770ff]{ss11}. It is highly non-trivial that these different approaches yield equivalent results. Camia and Newman established the full scaling limit of critical percolation on the triangular lattice as an ensemble of oriented loops, see \cite{cn6}. Schramm and Smirnov suggested to look at the set of quads which are crossed by the percolation configuration and constructed a nice topology for that purpose, the so-called Quad-Crossing-Topology, see \cite{ss11}. Since it is closely related to the original physical motivation of percolation and it yields the existence of limit points for free (by compactness), we choose to work with Schramm and Smirnov's set-up.

They considered percolation models on tilings of the plane, rather than on lattices. Each tile is either coloured blue or yellow, independently of each other. All site or bond percolation models can be handled in this way using appropriate tilings. The results of \cite{ss11} hold on a wide range of percolation models. In fact, two basic assumptions on the one-arm event and on the four-arm event are sufficient. The results of the present article also hold on rather general tilings, but a bit stronger assumptions are needed. Basically, we require the assumption of \cite{ss11} on the four-arm event and the Russo-Seymour-Welsh Theory (RSW). The exact conditions are presented below. In particular, we need the arm separation lemmas of \cite{k87} and \cite{n7}. They should hold on any graph which is invariant under reflection in one of the coordinate axes and under rotation around the origin by an angle $\phi\in(0,\pi)$, as stated in \cite[p.\ 112]{k87}. But the proofs are written up only for bond or site percolation on the square lattice in \cite{k87} and for site percolation on the triangular lattice in \cite{n7}. Hence we choose to formulate the exact properties we need as conditions. We will first prove our results under that conditions and we will verify them for bond percolation on the square lattice and site percolation on the triangular lattice afterwards.

We want to consider nearcritical scaling limits. Nearcritical percolation is obtained by colouring a tile blue with a probability slightly different from the critical one. The difference depends on the mesh size, but converges to zero in a well-chosen speed. It includes -- for each tile -- one free real parameter. The main result of the present note is the following: We consider two (inhomogeneous) nearcritical percolations such that the differences of their parameters are uniformly bounded away from zero in a macroscopic region. Then we show that any corresponding sub-sequential scaling limits are singular with respect to each other. 

Nolin and Werner showed in \cite[Proposition 6]{nw9} that -- on the triangular lattice -- any (sub-sequential) scaling limit of nearcritical exploration paths is singular with respect to an $\sle_6$ curve, i.e.\ to the limit of critical exploration paths. This was extended in \cite[Theorem 1]{a12}, where it is shown that the limits of two nearcritical exploration paths with different parameters are singular with respect to each other. The present result is somewhat different to those results, as we will now explain. First, we consider different objects. While in \cite{nw9} and \cite{a12} the singularity of exploration paths was detected, here it is the singularity of the full configurations in the Quad-Crossing-Topology. As long as the equivalence of the different descriptions of the limit object is not proven, these are independent results.  In particular, it is -- even on the triangular lattice -- an open question, whether the exploration path as a curve is a random variable of the set of all crossed quads (cf.\ \cite[Question 2.14]{gps10}). Though the trace of the exploration path can be recovered from the set of all crossed quads, it is not clear how to detect its behaviour at double points. Thus the present result is not an easy corollary to the singularity of the exploration paths. Second, the results of \cite{nw9} and \cite{a12} hold only for site percolation on the triangular lattice, whereas the results of the present article hold under rather general assumptions on the lattice, which are, for instance, also fulfilled by bond percolation on the square lattice. Last, and indeed least, the percolation may also be inhomogeneous here. Since the restriction to homogeneous percolation in \cite{nw9} and \cite{a12} has only technical, but not conceptual reasons, this is only a minor difference. 

The proofs use ideas from \cite{nw9} and \cite{a12}. In fact, the proofs of this article are technically simpler since there is no need to consider domains with fractal boundary. In section \ref{sec:results}, we formally introduce the model and state all theorems and lemmas, which will be proved in section \ref{sec:proofs}.

\section{Results} \label{sec:results}
As already mentioned, we use the set-up of \cite{ss11}. Therefore we consider percolation on tilings of the plane rather than on lattices. A tiling is a collection of polygonal, topologically closed tiles such that the tiles may intersect each other only at their boundary and such that their union is the whole plane. We further require that the tilings are locally finite, i.e.\ any bounded set contains only finitely many tiles,  and trivalent, i.e.\ any point belongs to at most three tiles.

For $\e>0$, let $H_\e$ be a locally finite trivalent tiling such that the diameter of each tile is at most $\e$. A percolation model is obtained by colouring every tile either blue or yellow. Some tiles may have a deterministic colour, while each tile $t\in H_\e'\subseteq H_\e$ is coloured randomly blue with some probability $\p(t)\in[0,1]$ and otherwise yellow, independently of each other. Any site or bond percolation model can be realized using such a tiling, cf.\ \cite[p.\ 1774f]{ss11}. Colouring some tiles deterministically ensures that the tiling is trivalent. For each $\e>0$, we therefore obtain the probability space
$$ \Big(\,\bar{\Om}_\e:=\{\text{blue,yellow}\}^{H_\e'},\quad \bar{\mathcal{A}}_\e\,, \quad \bar{P}^\p_\e := \bigotimes_{t\in H_\e'} \big( \p(t)\delta_\text{blue} + (1-\p(t))\delta_\text{yellow}\big)\Big) $$
with product-$\s$-algebra $\bar{\mathcal{A}}_\e$ and $\p:H_\e'\to[0,1]$. 

But we want to describe all discrete processes as well as the scaling limit by different probability measures on the same space. Thereto we use the space $\mathcal{H}$ of all closed lower sets of quads introduced by Schramm and Smirnov in \cite[Section 1.3]{ss11}. As the exact construction is not important for understanding the present note (but it is important for the properties derived in \cite{ss11} we need), we explain it only very briefly. A quad $Q$ is a homeomorphism $Q:[0,1]^2\to Q([0,1]^2)\subset\C$. A crossing of $Q$ is a connected closed subset of $Q([0,1]^2)$ which intersects the images of the left and the right side of $[0,1]^2$. The question, whether every crossing of a quad contains a crossing of a second quad, provides a partial order on the quads. If a set of quads also contains all smaller quads (in the sense of the partial order), it is called a lower set of quads. Then $\mathcal{H}$ is the space of all closed lower sets of quads. For a quad $Q$, we define the event $\hc{Q}\subset\mathcal{H}$ that the quad $Q$ is crossed: it is the set of all lower sets which contain $Q$. The space $\mathcal{H}$ is equipped with the so-called Quad-Crossing-Topology, which is the minimal topology containing all $(\hc{Q})^c$ and other certain lower sets of quads. The induced Borel-$\s$-algebra $\B(\mathcal{H})$ is generated by the events $\hc{Q}$. For $D\subset\C$, let $\B_D$ be the restriction of $\B(\mathcal{H})$ to lower sets of quads inside $D$.

Any configuration $\bar{\om}_\e\in\bar{\Om}_\e$ induces an element of $\mathcal{H}$, namely the set $\om_\e$ of all quads, which contain a blue crossing, i.e.\ a crossing which is a subset of the union of all blue tiles. Note that this is a closed lower set. Thus, for all $\e>0$ and $\p:H_\e'\to[0,1]$, the measure $\bar{P}^\p_\e$ induces a probability measure $P^\p_\e$ on $(\mathcal{H},\B(\mathcal{H}))$. We will mainly work with these probability measures.

Now we define a special measure on $\mathcal{H}$, namely the critical measure $P_\e^0$. It is induced by $\bar{P}^\p_\e$ with $\p(t)=p_\e^\text{crit}$ for all tiles $t\in H_\e'$. There $p_\e^\text{crit}$ is the critical probability of the tiling $H_\e$, i.e.\
$$ p_\e^\text{crit} := \sup\{p\in[0,1]\mid P^\p_\e[\text{There is an infinite blue cluster}]=0, \p(t)=p\,\forall t\in H_\e'\} \,.$$
In fact, we do not use criticality. Thus $p_\e^\text{crit}$ could be any number in $(0,1)$ such that the conditions below are satisfied. But they usually hold only if $p_\e^\text{crit}$ is indeed the critical probability.

For $z\in\C$ and $0<\e\le r<R$, let $A_4(z,r,R)$ be the event that there are four crossings of alternating colour inside the annulus centred at $z$ with radii $r$ and $R$. 

We fix some $R_0,N_0>0$ and $z_0\in\C$ for the remainder of the article. We want to define the nearcritical models. We abbreviate 
$$\al_\e^4 := P_\e^0[A_4(z_0,\e,R_0)]$$
and define the set
$$ \Pi_\e := \big\{ P^\p_\e \mid \p(t)=(p_\e^\text{crit}+\io_\e(t)\cdot\frac{\e^2}{\al_\e^4})\vee0\wedge1,\;\io_\e(t)\in[-N_0,N_0], \,t\in H_\e' \big\} \,,$$
the set of all probability measures on $(\mathcal{H},\B(\mathcal{H}))$ which are in the critical window. If we want to specify the chosen parameter $\io=(\io_\e(t))_{t\in H_\e'}$, we write $P^\io_\e$ for the corresponding measure. We therefore use the speed factor $\e^2/\al_\e^4$ for the convergence of the nearcritical  probabilities to the critical one. This rate is inspired by \cite[Theorem~4]{k87}, \cite[Proposition~32]{n7} and the results of \cite{gps10}. From Lemma~\ref{lem:characle} below and \cite[Proposition~4]{nw9}, it follows that $\e^2/\al_\e^4$ is indeed the correct rate.

\begin{cond}
We impose the following basic conditions on the tilings $H_\e$, $\e>0$. The constants $\e_0, c_1, c_2, c_3>0$ as well as the functions $\Delta_4$ and $\Delta_1$ may depend on $R_0$ and $N_0$. The words in \emph{italic} are only headings without any formal meaning. 
\begin{enumerate} 
 \item \label{en:4arm}
 \emph{The following multi-scale bound on the four arm event holds:} 
 
 There exists a positive function $\Delta_4(r,R)$ such that for all fixed $R\le R_0$
 $$\lim_{r\to0} \D_4(r,R) = 0$$
 and such that for all $\e\le r < R \le R_0$ 
 $$ P^0_\e[A_4(z_0,r,R)] \le \frac{r}{R}\D_4(r,R)\,. $$
 
 \item  \label{en:conv}
 \emph{The probabilities in the critical window are eventually strictly in between 0 and 1:}
 
 There exists $\e_0>0$ such that for all $\e\in(0,\e_0)$: 
 $$ 0 \,<\, p_\e^\text{crit} - N_0\frac{\e^2}{\al_\e^4} \,<\, p_\e^\text{crit} + N_0\frac{\e^2}{\al_\e^4} \,<\, 1 \,.$$
 
 \item \label{en:4armnc}
 \emph{The probabilities of the four-arm events are comparable on the whole plane over all (near)critical measures:} 
 
 There are constants $c_1,c_2>0$ such that for all $\e\le r< R \le  R_0$, $z\in\C$ and $P_\e\in\Pi_\e$ the following holds:
 $$
 \begin{array}{rcccl}
  c_1 \,P_\e^0[A_4(z_0,\e,R)] &\le& P_\e[A_4(z,\e,R)] &&\text{and}\\[1ex]
    && P_\e[A_4(z,r,R)] &\le& c_2\, P_\e^0[A_4(z_0,r,R)] \,.
 \end{array}
 $$
 (Note that we need the first inequality for $r=\e$ only.)

 \item  \label{en:sep}
 \emph{The probability of the four arm event is uniformly comparable to the probability of the following modified four arm event:} 
 
 For $R>0$ and $z\in\C$, let $Q(z,R)$ be the square with side length $R$ centred at $z$. For a tile $t$ in $Q(z,R)$ whose distance from $z$ is at most $R/4$, let  $A_4'(t,\del Q(z,R))$ be the event that there are four arms of alternating colour from $t$ to the left, lower, right and upper boundary of $Q(z,R)$, respectively.
 
 There exists a constant $c_3>0$ such that for all $4\e\le R \le R_0$, $z\in\C$, $P_\e\in\Pi_\e$ and all tiles $t$ in $Q(z,R)$ whose distance from $z$ is at most $R/4$:
 $$ P_\e[A_4'(t,\del Q(z,R))] \ge c_3 \,P_\e[A_4(z,\e,R)] \,.$$
 
 \item  \label{en:1arm}
 \emph{There is the following bound on the one arm event:}
 
 There exists a positive function $\Delta_1(r,R)$ such that for all fixed $R\le R_0$
 $$\lim_{r\to0} \D_1(r,R) = 0$$
 and such that for all $\e\le r < R \le R_0$, $z\in\C$, $P_\e\in\Pi_\e$ and $col\in\{\text{blue},\text{yellow}\}$
 $$ P_\e[A_1^{col}(z,r,R)] \le \D_1(r,R) \,,$$
 where $A_1^{col}(z,r,R)$ is the event that there exists a crossing of colour $col$ inside the annulus centred at $z$ with radii $r$ and $R$.
 \label{en:last}
\end{enumerate}
\end{cond}

Conditions \ref{en:4armnc} and \ref{en:4arm} imply
$$ P_\e[A_4(z,r,R)] \le \frac rR c_2 \Delta_4(r,R) $$
for all $z\in\C$ and $\e\le r< R \le R_0$ and $P_\e\in\Pi_\e$. This and condition \ref{en:1arm} are Assumptions 1.1.\ of \cite{ss11}. Therefore we can apply most results of that article, including \cite[Corollary 1.16]{ss11}, yielding that any family $P_\e\in\Pi_\e$, $\e>0$, is tight. Thus there exist nearcritical scaling limits, at least along subsequences.

Now we are ready to state the main theorem of the present note.

\begin{thm} \label{thm:fullsing}
 Let $H_\e$, $\e>0$, be locally finite trivalent tilings such that each tile of $H_\e$ has diameter at most $\e$, and such that conditions 1-\ref{en:last} are fulfilled. For $\e>0$, let measures $P^\mu_\e, P^\la_\e\in\Pi_\e$ be given by $\mu_\e(t),\la_\e(t)\in[-N_0,N_0]$, $t\in H_\e'$. Considering weak limits with respect to the Quad-Crossing-Topology, let $P^\mu$ be any weak limit point of $\{P^\mu_\e:\e>0\}$, let $P^\la$ be any weak limit point of $\{P^\la_\e:\e>0\}$ and let $\e_n$, $n\in\N$, be a sequence converging to zero such that $P^\mu_{\e_n}\to P^\mu$ and $P^\la_{\e_n}\to P^\la$ weakly as $n\to\infty$.
 
 Assume that there exist $\s>0$ and an open, non-empty set $D\subset\C$ such that
 $$ \la_\e(t) - \mu_\e(t) \,\ge\,\s $$
 uniformly in $\e\in\{\e_n:n\in\N\}$ and all tiles $t\in H_\e'$ which are contained in $D$. 
 
 Then the laws $P^\mu$ and $P^\la$ -- even restricted to $\B_D$ -- are singular with respect to each other.
\end{thm}

Similarly to \cite[Corollary 2]{a12}, we can even detect the asymmetry by only looking at an infinitesimal neighbourhood of a point inside $D$, more precisely:

\begin{corol} \label{cor:restrict}
 Let the conditions of Theorem~\ref{thm:fullsing} be fulfilled. Let $z\in D$. Let
 $ \B_z := \bigcap_{n\in\N} \B_{B_{1/n}(z)} $
 be the tail-$\s$-algebra of the restrictions of $\B(\mathcal{H})$ to lower sets of quads in the ball $B_{1/n}(z)$.
 
 Then the laws $P^\mu$ and $P^\la$ restricted to $\B_z$ are singular with respect to each other.
\end{corol}

We base the proof of Theorem \ref{thm:fullsing} on the following two lemmas. The first one is specific for the model. The second one is rather abstract to detect the singularity.
\begin{lem} \label{lem:estprob}
 Under the conditions of Theorem~\ref{thm:fullsing}, there exists a function $\D_\s:\R^+\to\R^+$ with $\D_\s(\de)\to0$ as $\de\to0$ such that for any square $Q$ of side length $\de\le R_0$ inside $D$:
 $$ P^\la[\hc{Q}]-P^\mu[\hc{Q}] \ge \frac{\de}{\D_\s(\de)} \,,$$
 where $\hc{Q}$ denotes the event that there exists a horizontal blue crossing of the square $Q$.
\end{lem}

\begin{lem} \label{lem:abstrsing}
 Let $P$ and $P'$ be two probability measures on a space $(\Om,\A)$. Let $a,b>0$ and let $(\D_n)_{n\in\N}$ be a positive sequence converging to infinity. Set $K_n:=\lceil an^2\rceil$, $n\in\N$. For large enough $n\in\N$, let $X^n_k$, $k\in\{1,\ldots,K_n\}$, be random variables which are uncorrelated in $k$ with respect to $P$ and $P'$, absolutely bounded by $b$, and satisfy
 $$ E_{P'}[X^n_k] - E_P[X^n_k] \ge \tfrac1n\D_n\,,\qquad k\in\{1,\ldots,K_n\}\,. $$
 Then $P$ and $P'$ are singular with respect to each other.
\end{lem}

Using results of \cite{sw1}, \cite[Appendix B]{ss11}, \cite{n7} and \cite{k87} as well as standard techniques, we can easily verify conditions 1-\ref{en:last} in the two most important cases:
\begin{lem} \label{lem:cond}
 Conditions 1 to \ref{en:last} are fulfilled by tilings representing site percolation on the triangular lattice or bond percolation on the square lattice.
\end{lem}
Thereto we will need the following converse of \cite[Proposition 32]{n7}, which estimates the characteristic length. For the remainder of this section, we consider site percolation on the triangular lattice or bond percolation on the square lattice, each with mesh size $1$. Let $p_c=\2$ be the critical probability. For $\ep\in(0,\2)$ and $p\in(0,1)$, let $L_\ep(p)$ be the corresponding characteristic length as defined in \cite[Section 3.1]{n7} or \cite[Equation (1.21)]{k87}, respectively, i.e.\
$$ L_\ep(p) := \begin{cases}
                \inf\{n\in\N: P_p[\hc{(n\times n)}] \le \ep\} & \text{ if }p<p_c \\
                \inf\{n\in\N: P_p[\hc{(n\times n)}] \ge 1-\ep\} & \text{ if }p>p_c
               \end{cases} $$
and $L_\ep(p_c)=\infty$, where $P_p$ denotes the product measure with probability $p$ for blue, and $\hc{(m\times n)}$ denotes the event that there is a horizontal blue crossing of a an $m\times n$ rectangle.
 
Moreover, for $m<n$, let $\al_4(m,n)$ be the probability that at critical percolation there exist four arms of alternating colour inside the annulus centred at the origin with radii $m$ and $n$. We abbreviate $\al_4(n):=\al_4(1,n)$.

\begin{lem} \label{lem:characle}
 For all $\ep\in(0,\2)$ and $C_1,C_2>0$, there exist $C_3,C_4>0$ such that for all $p\in(0,1)$ and $n\ge1$ the following implication holds:
 $$ C_1 \,\le\, |p-p_c| n^2 \al_4(n) \,\le\, C_2 \quad \Longrightarrow \quad C_3 \,\le\, \frac{n}{L_\ep(p)} \,\le\, C_4 \,.$$
\end{lem}

Finally, we need the following lemma, which restates Remark 36 of \cite{n7}. Since the author is not aware of a formal statement in the literature, it is included here for the sake of completeness.
\begin{lem} \label{lem:incrL}
 For all $\ep_0\in(0,\2)$ and all $K\ge1$, there exists an $\ep\in(0,\ep_0)$ such that for all $0<p<p_c$:
 $$ L_\ep(p) \ge K\cdot L_{\ep_0}(p) \,. $$
\end{lem}

\section{Proofs} \label{sec:proofs}

In this section, we give the proofs of all stated assertions.

\begin{proof}[Proof of Lemma \ref{lem:estprob}]
 Let $Q$ be a square of side length $\de\le R_0$ inside $D$. Let $\e\in\{\e_n:n\in\N\}$ be small enough such that $4\e<\de$ and $\e<\e_0$, where $\e_0$ is chosen according to condition \ref{en:conv}.
 
 We construct a coupling $(\hat{\Om},\hat{\A},\hat{P})$ as follows. Let
 $$ \hat{\Om} := \big(\{\text{blue,yellow}\}\times\{\text{blue,yellow}\}\big)^{H_\e'} $$
 with product-$\s$-algebra $\hat{\A}$. Informally, let $\hat{P}$ be the probability measure which has marginal distributions $\bar{P}^\mu_\e$ and $\bar{P}^\la_\e$ such that the set of blue tiles in $Q$ increases. More precisely, we define the random variables 
 $$f_I:\hat{\Om}\to\mathcal{H}\,,\quad I\in\{1,2\}^{H_\e'}\,.$$
 For $\hat{\om}=(\hat{\om}_1(t),\hat{\om}_2(t))_{t\in H_\e'}\in\hat{\Om}$, let $f_I(\hat{\om})$ be the set of all quads which contain a blue crossing if tile $t\in H_\e'$ is coloured with colour $\hat{\om}_{I(t)}(t)$. We abbreviate $\langle1\rangle:=(1,\ldots,1)$ and $\langle2\rangle:=(2,\ldots,2)$. Then let $\hat{P}$ be a probability measure on $\hat{\Om}$ such that
 $$ \begin{array}{c}
  f_{\langle1\rangle}\big(\hat{P}\big)=P^\mu_\e \,,\qquad f_{\langle2\rangle}\big(\hat{P}\big)=P^\la_\e \qquad\text{and} \\[1ex]
  \hat{P}\big[\hat{\om}: \hat{\om}_1(t)=\text{blue},\,\hat{\om}_2(t)=\text{yellow} \text{ for some tile $t$ in } Q\big] = 0 \,.
 \end{array} $$
 Such a coupling can be obtained, for example, from the standard monotone coupling using independent, uniformly on $[0,1]$ distributed random variables as $\la_\e(t)>\mu_\e(t)$ in $Q$. 
 
 It follows that
 \begin{eqnarray*}
  P^\la_\e[\hc{Q}]-P^\mu_\e[\hc{Q}]
  &=& \hat{P}\big[f_{\langle2\rangle}^{-1}[\hc{Q}]\setminus f_{\langle1\rangle}^{-1}[\hc{Q}]\big]
     -\hat{P}\big[f_{\langle1\rangle}^{-1}[\hc{Q}]\setminus f_{\langle2\rangle}^{-1}[\hc{Q}]\big]\\
  &=& \hat{P}\big[f_{\langle1\rangle}^{-1}[\hc{Q}]^c \cap f_{\langle2\rangle}^{-1}[\hc{Q}]\big] - 0 \,, 
 \end{eqnarray*}
 since $\hat{\om}\in f_{\langle1\rangle}^{-1}[\hc{Q}]\setminus f_{\langle2\rangle}^{-1}[\hc{Q}]$ implies that there is a tile $t$ in $Q$ with $\hat{\om}_1(t)=\text{blue}$ and $\hat{\om}_2(t)=\text{yellow}$. Thus we have to estimate the probability of the event of all $\hat{\om}=(\hat{\om}_1,\hat{\om}_2)$ such that $\hat{\om}_2$ induces a blue crossing of $Q$, but $\hat{\om}_1$ does not.
 
 Let $\T=\{t_1,\ldots,t_K\}$ be the set of all tiles in $Q$ whose distance from the centre $z_Q$ of $Q$ is at most $\de/4$ -- arranged in any (but fixed) order. In order to prove the proposed estimate, we restrict ourselves to the event that the crossing arises out of switches from yellow to blue of some tiles in $\T$. Thereto we change the coordinates of $\hat{\om}$ we use for the tiles in $\T$ one by one. Formally, for $k=0,\ldots,K$, let $I_k\in\{1,2\}^{H_\e'}$ be defined by $I_k(t)=1$ if $t\in H_\e'\setminus\{t_1,\ldots,t_k\}$, and $I_k(t)=2$ if $t\in\{t_1,\ldots,t_k\}$. Then
 $$ \hat{P}\big[f_{\langle1\rangle}^{-1}[\hc{Q}]^c \cap f_{\langle2\rangle}^{-1}[\hc{Q}]\big]
    \ge \hat{P}\big[\bigcup_{k=1}^K f_{I_{k-1}}^{-1}[\hc{Q}]^c \cap f_{I_k}^{-1}[\hc{Q}]\big] \,.$$
 As the crossing event is increasing, the event $f_{I_{k-1}}^{-1}[\hc{Q}]^c \cap f_{I_k}^{-1}[\hc{Q}]$ can happen only for one $k\in\{1,\ldots,K\}$. This is the case if and only if the following two events occur: first, the event $f_{I_k}^{-1}[A_4'(t_k,\del Q)]$ that there are four arms of alternating colour from $t_k$ to the left, lower, right and upper boundary of $Q$, respectively, which means that $t_k$ is pivotal for the crossing event; second, the event that the colour of $t_k$ switches from $\hat{\om}_1(t_k)=$ yellow to $\hat{\om}_2(t_k)=$ blue, which we denote by $Sw(t_k)$. Note that they are independent events. Using the described disjointness and independence, we get
 $$ \hat{P}\big[\bigcup_{k=1}^K f_{I_{k-1}}^{-1}[\hc{Q}]^c \cap f_{I_k}^{-1}[\hc{Q}]\big]
    = \sum_{k=1}^K \hat{P}\big[f_{I_k}^{-1}[A_4'(t_k,\del Q)]\big] \cdot \hat{P}[Sw(t_k)] \,.$$
 
 Now we estimate these probabilities. Elementary probability calculus and the construction of the coupling yield
 \begin{eqnarray*}
  \hat{P}[Sw(t_k)]
  &=& \hat{P}\big[\{\hat{\om}: \hat{\om}_2(t_k)=\text{blue}\}\setminus\{\hat{\om}: \hat{\om}_1(t_k)=\text{blue}\}\big] \\
  &=& \hat{P}\big[\{\hat{\om}: \hat{\om}_2(t_k)=\text{blue}\}\big] - \hat{P}\big[\{\hat{\om}: \hat{\om}_1(t_k)=\text{blue}\}\big] + \\
  &&\quad + \,\hat{P}\big[\{\hat{\om}: \hat{\om}_1(t_k)=\text{blue}\}\setminus\{\hat{\om}: \hat{\om}_2(t_k)=\text{blue}\}\big]\\
  &=&  P^\la_\e[t_k\text{ blue}] - P^\mu_\e[t_k\text{ blue}] + \hat{P}[\hat{\om}: \hat{\om}_1(t_k)=\text{blue},\,\hat{\om}_2(t_k)=\text{yellow}] \\
  &=&  \big(p_\e^\text{crit}+\la_\e(t_k)\cdot\frac{\e^2}{\al_\e^4}\big) - \big(p_\e^\text{crit}+\mu_\e(t_k)\cdot\frac{\e^2}{\al_\e^4}\big) + 0\\
  &=&  \big(\la_\e(t_k)-\mu_\e(t_k)\big)\frac{\e^2}{\al_\e^4} \,\ge\, \s \cdot \frac{\e^2}{\al_\e^4} \,,
 \end{eqnarray*}
 because of $\e<\e_0$ (such that, by condition \ref{en:conv},  the probabilities are given by the used formulas) and because of the assumption in Theorem~\ref{thm:fullsing}. 
 
 Let $P^{I_k}_\e$ denote the image law of $\hat{P}$ under $f_{I_k}$. Then $P^{I_k}_\e\in\Pi_\e$. Using conditions \ref{en:sep} and \ref{en:4armnc}, we conclude
 $$ \hat{P}\big[f_{I_k}^{-1}[A_4'(t_k,\del Q)]\big] \,\ge\, c_3 P^{I_k}_\e[A_4(z_Q,\e,\delta)] \,\ge\, c_3c_1 P^0_\e[A_4(z_0,\e,\de)] \,.$$
  
 As there are $K\ge c_4(\de/\e)^2$ tiles in $\T$ (for some numerical constant $c_4>0$), the equations above imply
 $$ P^\la_\e[\hc{Q}]-P^\mu_\e[\hc{Q}] \ge c_4\frac{\de^2}{\e^2}\cdot c_3c_1 P^0_\e[A_4(z_0,\e,\de)] \cdot \s \frac{\e^2}{\al_\e^4} = \s c_1c_3c_4\cdot \frac{\de^2 P^0_\e[A_4(z_0,\e,\de)]}{P^0_\e[A_4(z_0,\e,R_0)]}\,. $$
 Using first $A_4(z_0,\e,R_0)\subseteq A_4(z_0,\e,\de)\cap A_4(z_0,\de,R_0)$ and independence of the latter two events and then condition \ref{en:4arm}, we conclude
 \begin{eqnarray*}
  P^\la_\e[\hc{Q}]-P^\mu_\e[\hc{Q}]
  &\ge& \s c_1c_3c_4\cdot \frac{\de^2}{P^0_\e[A_4(z_0,\de,R_0)]}\\
  &\ge& \s c_1c_3c_4\cdot \frac{\de^2 R_0}{\de \D_4(\de,R_0)} = \frac{\de}{\D_\s(\de)}  
 \end{eqnarray*}
 with $\D_\s(\de):=(\s c_1c_3c_4R_0)^{-1} \D_4(\de,R_0)$. Condition \ref{en:4arm} implies $\D_\s(\de)\to0$ as $\de\to0$.
 
 For $\io\in\{\mu,\la\}$, Lemma 5.1 of \cite{ss11} (implying $P^\io[\del\hc{Q}]=0$) and the weak convergence of $P^\io_{\e_n}$ yield $P^\io_{\e_n}[\hc{Q}]\to P^\io[\hc{Q}]$ as $n\to\infty$, which concludes the proof.
\end{proof}

\begin{proof}[Proof of Lemma \ref{lem:abstrsing}]
 We define for large enough $n\in\N$
 $$ Z_n := \sum_{k=1}^{K_n} \big( X^n_k-E_P[X^n_k]\big) \,.$$
 It follows that $E_P[Z_n]=0$ and that
 $$ E_{P'}[Z_n] \,=\, \sum_{k=1}^{K_n}\big(E_{P'}[X^n_k] - E_P[X^n_k]\big) \,\ge\, K_n\cdot\tfrac1n\D_n\,\ge\, an\D_n \,,$$
 because of the assumption and $K_n=\lceil an^2\rceil$. Since the random variables are uncorrelated and bounded, we can estimate the variance of $Z_n$ under $P$ or under $P'$ as follows:
 $$ \Var[Z_n]=\sum_{k=1}^{K_n}\Var\big[X^n_k-E_P[X^n_k]\big] = \sum_{k=1}^{K_n}\Var[X^n_k]\le K_nb^2\le (a+1)b^2n^2 \,.$$
 Using Chebyshev's Inequality, we estimate
 $$ P[Z_n\ge\tfrac a2n\D_n] \,\le\, \frac{4}{a^2n^2\D_n^2} \Var_P[Z_n]\,\le\, \frac{4(a+1)b^2n^2}{a^2n^2\D_n^2}\,=\,\frac{4(a+1)b^2}{a^2}\cdot \D_n^{-2} $$
 and
 \begin{eqnarray*}
  P'[Z_n<\tfrac a2n\D_n]
  &=& P'\big[(E_{P'}[Z_n]-Z_n)>(E_{P'}[Z_n]-\tfrac a2n\D_n)\big] \\
  &\le& P'\big[|E_{P'}[Z_n]-Z_n|>(an\D_n-\tfrac a2n\D_n)\big] \\
  &\le& \frac{4}{a^2n^2\D_n^2} \Var_{P'}[Z_n] \,\le\, \frac{4(a+1)b^2}{a^2}\cdot \D_n^{-2}\,.
 \end{eqnarray*}
If we now choose a sparse enough sub-sequence $n_l$, $l\in\N$, i.e.\ such that $\sum_l\D_{n_l}^{-2}<\infty$, the Borel-Cantelli Lemma yields
$$ \begin{array}{rrcl}
 & P'\big[Z_{n_l}<\tfrac a2n_l\D_{n_l}\text{ for infinitely many }l] &=& 0 \\[1ex]
 \text{implying }& P'\big[Z_{n_l}\ge\tfrac a2n_l\D_{n_l}\text{ for infinitely many }l] &=& 1\,,\\[1ex]
 \text{while }& P\big[Z_{n_l}\ge\tfrac a2n_l\D_{n_l}\text{ for infinitely many }l] &=& 0\,.
\end{array} $$
Therefore we detected an event which has $P$-probability zero, but $P'$-probability one.
\end{proof}

\begin{proof}[Proof of Theorem~\ref{thm:fullsing}.]
 We want to apply Lemma~\ref{lem:abstrsing}. Let $P'=P^\la$ and $P=P^\mu$. We set $\de_n=1/n$, $n\in\N$, and choose an appropriate $a>0$ (depending on the size of $D$) such that, for sufficiently large $n$, we can place $K_n=\lceil an^2\rceil$ disjoint squares $Q^n_1,\ldots,Q^n_{K_n}$ of size $\de_n$ in $D$.  We define the random variables $ X^n_k:\HH\to\R$ by
 $$ X^n_k = \1_{\hc{Q^n_k}}\,,\qquad k\in\{1,\ldots,K_n\}\,. $$
 Since the disjointness of the squares yields independence of the crossing events for all $P^\io_{\e_m}$, since $P^\io_{\e_m}\to P^\io$ weakly and since $P^\io[\del\hc{Q^n_k}]=0$ by \cite[Lemma~5.1]{ss11}, the random variables $X^n_k$, $k\in\{1,\ldots,K_n\}$, are independent for $P^\io$, $\io\in\{\mu,\la\}$. Moreover, $|X^n_k|\le1$, and Lemma \ref{lem:estprob} yields
 $$ E_{P^\la}[X^n_k] - E_{P^\mu}[X^n_k] = P^\la[\hc{Q^n_k}]-P^\mu[\hc{Q^n_k}] \ge \frac{\de_n}{\D_\s(\de_n)}=\tfrac1n\D_n $$
 with $\D_n:=\D_\s(\de_n)^{-1}\to\infty$ as $n\to\infty$. Thus Lemma~\ref{lem:abstrsing} yields that $P^\mu$ and $P^\la$ are singular with respect to each other. Since all random variables $X^n_k$ are $\B_D$-measurable, we can also apply Lemma~\ref{lem:abstrsing} when $P^\mu$ and $P^\la$ are restricted to $\B_D$.
\end{proof}

\begin{proof}[Proof of Corollary \ref{cor:restrict}]
 The proof is analogous to the proof of \cite[Corollary 2]{a12}. Let $m_0\in\N$ such that $B_{\frac{1}{m_0}}(z)\subseteq D$. Let $n\ge m_0$. By Theorem~\ref{thm:fullsing} -- applied inside $B_{\frac1n}(z)$ -- there are sets $B_n\in\B_{B_{\frac1n}(z)}$ with $P^\mu[B_n]=0$ and $P^\la[B_n]=1$. We set
 $$ B_* := \bigcup_{m\ge m_0} \bigcap_{n\ge m} B_n \,.$$
 Then $B_*\in\B_z$. Since countable unions or intersection of sets of probability zero respectively one have probability zero respectively one, it follows that $P^\mu[B_*]=0$ and $P^\la[B_*]=1$, which proves the corollary.
\end{proof}

\begin{proof}[Proof of Lemma \ref{lem:cond}]
 As it is proven on the triangular lattice that the 4-arm-exponent is $5/4$, see \cite[Theorem 4]{sw1}, condition \ref{en:4arm} holds. For bond percolation on the square lattice, this condition is proven by Christophe Garban in \cite[Lemma B.1]{ss11}.
 
 Now we claim that $R_0$ is below a characteristic length of $p^{-N_0}_\e=p_\e^\text{crit}-N_0\e^2/\al_\e^4$, i.e.\ there is some $\ep\in(0,\2)$ such that $R_0/\e \le L_\ep(p^{-N_0}_\e)$ for all $\e>0$. Thereto we provisionally fix some $\ep_0\in(0,\2)$. Since
 $$ \big|p^{-N_0}_\e - p_\e^\text{crit}\big|\, (R_0/\e)^2 P_\e^0[A_4(z_0,\e,R_0)] \,=\, R_0^2N_0 \,,$$
 Lemma~\ref{lem:characle} (for $n=(R_0/\e)$ and $p=p^{-N_0}_\e$) yields that $R_0/\e \le C_4 L_{\ep_0}(p^{-N_0}_\e)$ for some $C_4=C_4(R_0,N_0,\ep_0)>0$. By Lemma~\ref{lem:incrL}, we find an $\ep\in(0,\ep_0)$ such that the claim holds.
 
 Now we fix this $\ep>0$. Since every $P_\e\in\Pi_\e$ is between $P^{-N_0}_\e$ and $P^{+N_0}_\e$, the claim above allows us to use arguments of RSW style and to apply most of the results of \cite{n7} and \cite{k87} as long as we use radii $R\le R_0$. In fact, all of the remaining conditions easily follow from the results of these papers.
   
 The following reasoning is a standard technique. By RSW, there is a constant $c>0$ such that for all $col\in\{\text{blue},\text{yellow}\}$, $z\in\C$, $\e\le r \le R_0/2$ and $P_\e\in\Pi_\e$
 $$ c \le P_\e[A_1^{col}(z,r,2r)] \le 1-c \,.$$
 Let $R\le R_0$ be fixed. For $r\in(\e,R/2)$, let $K_r\in\N$ be the largest number such that $2^{K_r}\le R/r$. Then $K_r\to\infty$ as $r\to0$. It follows that
 \begin{eqnarray*}
  P_\e[A_1^{col}(z,r,R)] &\le& P_\e[\forall k=1,\ldots,K_r:\, A_1^{col}(z,r2^{k-1},r2^k)] \\
  &\le&  \prod_{k=1}^{K_r} P_\e[A_1^{col}(z,r2^{k-1},r2^k)] \le (1-c)^{K_r} \to 0
 \end{eqnarray*}
 as $r\to0$, which shows condition \ref{en:1arm} (on both lattices).
 
 By \cite[Corollary A.8]{ss10} (stating that the 5-arm-exponent is 2) and Reimer's Inequality, it follows that (for some $\ti{c}>0$)
 \begin{equation}
 \ti{c}R_0^{-2}\e^2 \le P_\e^0[A_5(z_0,\e,R_0)] \le P_\e^0[A_4(z_0,\e,R_0)] \cdot P_\e^0[A_1(z_0,\e,R_0)] \,.
 \tag{1}
 \end{equation}
 Thus condition \ref{en:1arm} yields $\e^2/\al_\e^4 \to0 $ as $\e\to0$, which, together with $p_\e^\text{crit}=\2$, implies condition \ref{en:conv} (on both lattices). 
 
 Since the considered lattices are transitive, the estimates in conditions \ref{en:4armnc} and \ref{en:sep} hold uniformly in $z\in\C$, if they hold for $z=0$. Thus we consider only this case. 
 Condition \ref{en:4armnc} on the triangular lattice is included in Theorem 26 of \cite{n7}. On the square lattice, condition \ref{en:4armnc} is a consequence of \cite[Lemma 8]{k87} (with $v=0$) and \cite[Lemma 4]{k87}.  
 These two lemmas (with $\kappa=0.5$) also imply condition \ref{en:sep} on the square lattice. On the triangular lattice, it is a special case of equation (4.20) in \cite{n7}. 
\end{proof}

Note that we are considering site percolation on the triangular lattice or bond percolation on the square lattice with mesh size 1 in the remaining two lemmas.

\begin{proof}[Proof of Lemma~\ref{lem:characle}]
 We fix some $\ep\in(0,\2)$ and abbreviate $L(p):=L_\ep(p)$. We will use the following facts. First,  
 \begin{equation}
  \exists\,\tid{C_1}(\ep),\tid{C_2}(\ep)>0  \,\forall\, p\in(0,1) :\; \tid{C_1} \,\le\, |p-p_c| L(p)^2\, \al_4\big(L(p)\big) \,\le\, \tid{C_2} \,,\tag{i}
 \end{equation}
 which is \cite[Proposition 32]{n7} for the triangular lattice and \cite[Theorem 4]{k87} for the square lattice. Second, we need quasi-multiplicativity  \cite[Proposition 4]{ss10}:
 \begin{equation}
  \exists\,C_5>0 \,\forall\, m<\ti{n} :\; \al_4(m)\cdot\al_4(m,\ti{n}) \,\le\, C_5\,\al_4(\ti{n})\,. \tag{ii}
 \end{equation}
 Finally, we need an estimate of the four arm event, namely
 \begin{equation}
  \exists\,\beta,C_6>0 \,\forall\, m<\ti{n} :\; \al_4(m,\ti{n}) \,\ge\, C_6\Big(\frac{m}{\ti{n}}\Big)^{2-\beta}\,. \tag{iii}
 \end{equation}
 Its proof is analogous to the proof of equation (1) above. Note that we can a-priori apply the RSW theory for (iii), since there we consider only critical percolation. 
 
 Let $C_1,C_2>0$. We define $C_3,C_4>0$ by
 $$ C_4 \,:=\, \max\big\{\big(\tfrac{C_2C_5}{\tid{C_1}C_6}\big)^{\frac1\beta},1\big\} \quad\text{and}\quad
    \frac{1}{C_3} \,:=\, \max\big\{\big(\tfrac{\tid{C_2}C_5}{C_1C_6}\big)^{\frac1\beta},1\big\}\,. $$
 Let $p\in(0,1)$ and $n\ge1$ with
 \begin{equation}
  C_1 \,\le\, |p-p_c|\,n^2 \al_4(n) \,\le\,C_2 \,. \tag{$*$}
 \end{equation}

 First, we show that $n/L(p)\le C_4$. We can assume $n>L(p)$, since otherwise $n/L(p)\le 1 \le C_4$. Facts (ii) and (iii) with $m=L(p)$ and $\ti{n}=n$ imply
 $$ \frac{\al_4(n)}{\al_4(L(p))} \ge \tfrac{1}{C_5}\, \al_4\big(L(p),n\big) \ge \tfrac{C_6}{C_5} \Big(\frac{L(p)}{n}\Big)^{2-\beta} \,.$$
 Combined with the left inequality of (i) and the right inequality of ($*$), we conclude
 $$ \frac{C_2}{\tid{C_1}} \,\ge\, \frac{|p-p_c|\,n^2 \al_4(n)}{|p-p_c|\,L(p)^2\al_4(L(p))} \,\ge\, \Big(\frac{n}{L(p)}\Big)^2 \tfrac{C_6}{C_5} \Big(\frac{L(p)}{n}\Big)^{2-\beta} \,=\, \tfrac{C_6}{C_5} \Big(\frac{n}{L(p)}\Big)^\beta $$
 and therefore $n/L(p)\le C_4$.
 
 An analogous reasoning with interchanged roles of $L(p)$ and $n$ yields the other estimate, i.e.\ $L(p)/n\le1/C_3$. Thereto we may assume $L(p)>n$, since otherwise $L(p)/n\le1\le1/C_3$. Using facts (ii) and (iii) with $m=n$ and $\ti{n}=L(p)$, we get
 $$ \frac{\al_4(L(p))}{\al_4(n)} \ge \tfrac{1}{C_5}\, \al_4\big(n,L(p)\big) \ge \tfrac{C_6}{C_5} \Big(\frac{n}{L(p)}\Big)^{2-\beta} \,.$$
 Now we apply the right inequality of (i) and the left inequality of ($*$) to conclude
 $$ \frac{\tid{C_2}}{C_1} \,\ge\, \frac{|p-p_c|\,L(p)^2 \al_4(L(p))}{|p-p_c|\,n^2\al_4(n)} \,\ge\, \Big(\frac{L(p)}{n}\Big)^2 \tfrac{C_6}{C_5} \Big(\frac{n}{L(p)}\Big)^{2-\beta} \,=\, \tfrac{C_6}{C_5} \Big(\frac{L(p)}{n}\Big)^\beta $$
 and therefore $L(p)/n\le1/C_3$.
\end{proof}

\begin{proof}[Proof of Lemma~\ref{lem:incrL}]
 Let $\ep_0\in(0,\2)$ and $K\ge2$. The RSW Theorem (see \cite[Theorem 2]{n7}, for instance) states that there is a universal positive function $f_K(\cdot)$, such that, for all $m\in\N$, if the probability of crossing an $m\times m$ rectangle is at least $\delta$, then the probability of crossing an $Km\times m$ rectangle is at least $f_K(\delta)$. We set $\ep:=f_K(\ep_0)/2$. Then $\ep\in(0,\ep_0)$ as $f_k(\delta)\le\delta$. Let $p\in(0,p_c)$. We abbreviate $L:=L_{\ep_0}(p)$. We have to show that $L_\ep(p)\ge KL$. By the definition of $L_\ep(p)$, it suffices to show that $P_p[\hc{(n\times n)}]>\ep$ for all $n< KL$. If $n\le L$, then $P_p[\hc{(n\times n)}]\ge\ep_0>\ep$ by the definition of $L=L_{\ep_0}(p)$. Now let $n\in(L,KL)$. Since every crossing of a $KL\times L$ rectangle induces a crossing of an $n\times n$ rectangle (if the rectangles are matched on the upper left corner), it follows that
 $ P_p[\hc{(n\times n)}] \ge P_p[\hc{(KL\times L)}] \ge f_K(\ep_0) > \ep$,
 which completes the proof.
\end{proof}

\emph{Acknowledgement:}
This research was supported by a scholarship of the Cusanuswerk, one of the German national academic foundations.

\end{document}